\documentclass[10pt,leqno]{amsart}
\usepackage{amsmath}
\usepackage{amsfonts}
\usepackage{amssymb}
\usepackage{mathtools}
\usepackage{graphicx}
\usepackage{color}
\usepackage{hyperref}
\allowdisplaybreaks[4]

\usepackage{xcolor}
\newtheorem{theorem}{Theorem}[section]
\newtheorem*{theorem*}{Theorem}
\newtheorem{lemma}[theorem]{Lemma}
\newtheorem{corollary}[theorem]{Corollary}
\newtheorem{proposition}[theorem]{Proposition}

\theoremstyle{definition}
\newtheorem{definition}[theorem]{Definition}

\newtheorem{remark}{Remark}[section]

\def \b {\beta}

\def\Ric{\text{Ric}}

\def\a{\alpha}
\def\l{\lambda}

\def\S{{\operatorname{Scal}}}

\def\Ric{\operatorname{Ric}}

\numberwithin{equation}{section}
\makeatletter
\newcommand*\owedge{\mathpalette\@owedge\relax}
\newcommand*\@owedge[1]{%
  \mathbin{%
    \ooalign{%
      $#1\m@th\bigcirc$\cr
      \hidewidth$#1\m@th\wedge$\hidewidth\cr
    }%
  }%
}
\makeatother

\begin{document}

\title[Secondary Curvature Operator on Einstein Manifolds]{Cone
Conditions for  the Curvature Operator of the Second Kind on Einstein Manifolds
}

\author[Cheng]{Haiqing Cheng}
\address{School of Mathematical Sciences, Soochow University, Suzhou, 215006, China}
\email{chq4523@163.com}

\author[Wang]{Kui Wang}\thanks{}
\address{School of Mathematical Sciences, Soochow University, Suzhou, 215006, China}
\email{kuiwang@suda.edu.cn}



\begin{abstract}
In this note, we study  Einstein manifolds whose curvature operator of the second kind $\mathring{R}$ satisfies the cone condition
\[
\alpha^{-1}\big(\sum_{i=1}^{[\alpha]}
\lambda_i+ (\alpha - [\alpha] ) \lambda_{[\alpha] + 1} \big) \ge -\theta \bar{\lambda}
\]
for some real number $\alpha \in [1, (n+2)(n-1)/2)$. Here $[\alpha] :=\max\{ m \in \mathbb{Z}: m \leq \a\}$, $\theta>-1$ and $\lambda_1 \le \cdots \le \lambda_{(n+2)(n-1)/2}$ are the eigenvalues of $\mathring{R}$ and $\bar{\lambda}$ is their average. The main result states that any closed Einstein manifold of dimension $n \ge 4$ with $\mathring{R}$ satisfies the cone condition is flat or a round sphere. These results generalize recent works corresponding to $\alpha \in \mathbb Z_+$ of the authors \cite{CW24-1,CW25-2} and Fu-Lu \cite{FL25}.
\end{abstract}
\maketitle

\section{Introduction and main results}
The study of sphere theorems under curvature conditions has been a central theme in Riemannian geometry. In 1986, Nishikawa \cite{Nishikawa86} conjectured that a closed Riemannian manifold with positive (respectively, nonnegative) curvature operator of the second kind, denoted $\mathring{R}$ (see Section \ref{sec 2.2}), is diffeomorphic to a spherical space form (respectively, a Riemannian locally symmetric space). This conjecture has been resolved and greatly refined in recent years, leading to a systematic framework for classifying manifolds under various conditions on $\mathring{R}$. For related work,  see also \cite{CGT23, CW24-1, CW25-2, DF24, Kashiwada93, Li22JGA, Li22Kahler, Li22PAMS, Li21, Li22product, Li25, NPW22, NPWW22}.

For Riemannian manifolds, 
the positive case of Nishikawa's conjecture was resolved by Cao, Gursky, and Tran \cite{CGT23}, who showed that manifolds with $2$-positive curvature operators of the second kind is diffeomorphic to a spherical space form. Shortly after, Li \cite{Li21} weakened the assumption to $3$‑positive and classified closed manifolds with $3$‑nonnegative curvature operator of the second kind. Subsequently, Nienhaus, Petersen, and Wink \cite{NPW22} derived a Bochner formula for the curvature operator of the second kind and proved that $\frac{n+2}{2}$-nonnegative implies the manifold is either flat or a rational homology sphere. 

For Einstein manifolds,  Kashiwada \cite{Kashiwada93} showed that closed Einstein manifolds with positive curvature operator of the second kind are spheres. 
Cao-Gursky-Tran \cite{CGT23} proved that for Einstein manifolds, four-positive (respectively, four-nonnegative) curvature operator of the second kind implies constant sectional curvature (respectively, local symmetry). 
Li \cite{Li22JGA} generalized this to $4\frac12$-positive (respectively, $4\frac12$-nonnegative) curvature operator of the second kind.
By developing a Bochner formula for the curvature of the second kind, Nienhaus-Petersen-Wink \cite{NPW22} proved that any $n$-dimensional compact Einstein manifold with $\ell$-nonnegative ($ \ell< \frac{3n(n+2)}{2(n+4)}$) curvature operator of the second kind is either flat or a rational homology sphere.
Recently, Dai and Fu \cite{DF24} employed a Bochner technique to prove that a closed Einstein manifold of dimension $n \ge 4$ with nonnegative curvature operator of the second kind is a constant curvature space; more precisely, they showed that it suffices to assume the curvature operator is $2$-nonnegative if $n = 4$ or $8 \le n \le 10$, $3$-nonnegative if $n = 5$, and $[\frac{n+2}{4}]$-nonnegative if $n \ge 11$.

Recently, Li \cite{Li25} introduced a cone condition for the curvature operator of the second kind. Let $(M^n, g)$ be a Riemannian manifold, and denote by $\mathring{R}: S_0^2(TM) \to S_0^2(TM)$ its curvature operator of the second kind. Let $\lambda_1 \le \cdots \le \lambda_N$ be the eigenvalues of $\mathring{R}$ and let $\bar{\lambda} = \frac{1}{N}\operatorname{tr}(\mathring{R})$ be their average, where $N = \frac{(n-1)(n+2)}{2}$. For parameters $\alpha \in [1, N)$ and $\theta > -1$, we say that $\mathring{R}$ lies in the cone $\mathcal{C}(\alpha, \theta)$ if the following inequality holds:
\begin{align}\label{cone-cond}
\alpha^{-1}\big(\sum_{i=1}^{[\alpha]}
\lambda_i+ (\alpha - [\alpha] ) \lambda_{[\alpha] + 1} \big) \ge -\theta \bar{\lambda},
\end{align}
where $[\a]$ denotes the greatest integer less than or equal to $\alpha$.

In particular, when $\theta = 0$, the condition $\mathring{R} \in \mathcal{C}(\alpha, 0)$ reduces to the $\alpha$-nonnegative curvature operator of the second kind. Thus, the family $\mathcal{C}(\alpha, \theta)$ constitutes an interpolation between the familiar $\alpha$-nonnegativity and more flexible pinching-type bounds. 



For Riemannian manifolds, in dimensions three and four, Li \cite{Li25} proved that if $\mathring{R}$ lies in the interior of $\mathcal{C}(\alpha, \bar{\Theta})$, the manifold is diffeomorphic to a spherical space form, where $\bar{\Theta}$ is a positive number depends on $n$ and $\alpha$;
in higher dimensions ($n\geq5$), the condition $\mathring{R}\in\mathcal{C}\left(\frac{n+2}{2}, \bar\theta\right)$ with $-1<\bar\theta<\frac{2}{n+2}$ implies the manifold is either flat or a rational homology sphere.

For Einstein manifolds, several recent works have studied the case $\mathring{R}\in\mathcal{C}(\alpha,\theta)$ with $\theta>0$.
Li \cite{Li25} proved that the condition $\mathring{R}\in\mathcal{C}\left(\frac{n+2}{2}, \frac{2(n-1)}{n+2}\right)$ implies that the manifold is either flat or a rational homology sphere, and the constant $\frac{2(n-1)}{n+2}$ is optimal.
In \cite{CW24-1}, the authors of the present paper studied closed Einstein manifolds satisfying $\mathring{R} \in \mathcal{C}(1, \theta)$ and showed that, under suitable parameter constraints, such a manifold is flat or a round sphere. This condition was later extended to $\mathring{R} \in \mathcal{C}(2, \theta)$ in \cite{CW25-2} with analogous rigidity conclusions. Subsequently, Dai and Lu \cite{FL25} generalized these results to integer values $\alpha \ge 3$ (up to $\alpha \le \big[\tfrac{n+2}{4}\big]$), establishing cone conditions under which the manifold is necessarily either flat or a spherical space form.

In this paper, we investigate sphere theorems for Einstein manifolds satisfying the cone condition $\mathring{R}\in \mathcal{C}(\alpha,\theta)$ with real $\alpha$, and establish the following results.



\begin{theorem}\label{thm1}
Let $(M^n, g)$ be a closed Einstein manifold of dimension $n \ge 6$, and let $\mathring{R}$ be the curvature operator of the second kind. If $\mathring{R} \in \mathcal{C}(\alpha,\theta(n,\alpha) )$ for  $1 \le \alpha \le \min\left\{\frac{n^4-n^3+8 n-8}{3 n^3+5 n^2-22 n+8}, \frac{n^2+n-8}{4n-8} \right\}$, where 
\begin{align}\label{theta}
\theta(n,\alpha)
= \frac{3(N-n+1)(N-\alpha)}{3n\alpha(N-2)+(N-3)(N-\alpha)}-1,
\end{align}
then $M$ is flat or a round sphere. Here $N=\frac{(n+2)(n-1)}{2}$.
\end{theorem}



\begin{theorem}\label{thm2}
Let $(M^n, g)$ be a closed Einstein manifold of dimension $n=4,5$, and let $\mathring{R}$ be the curvature operator of the second kind. If $\mathring{R}$ satisfies
\[
\mathring{R}\in \mathcal{C}\bigg(\alpha, \frac{(n-1)\big((n+2)(n+5)-(3n+8)\alpha\big)}{3\a (n+3)(n-2)} \bigg)
\]
for some real number $1\le \a \le \frac{(n+2)(n+5)}{3n+8}$, then $M$ is flat or a round sphere.

\end{theorem}

The corollary below follows directly from Theorems \ref{thm1} and \ref{thm2}.

\begin{corollary}\label{cor1}
Let $(M^n, g)$ be a closed Einstein manifold of dimension $n \ge 4$. Then $M$ is flat or a round sphere if one of the following conditions holds:
\begin{itemize}
    \item[(1)] If $n=4$, and $\mathring{R} \in \mathcal{C}(\alpha,\frac{27-10\alpha}{7 \alpha})$ for some $\alpha \in (2,\frac{27}{10} ]$; 
    \item[(2)] If $n=5$, and $\mathring{R} \in \mathcal{C}(\alpha,\frac{70-23\alpha}{18\alpha} )$ for some $\alpha \in (3,\frac{70}{23}]$;
    \item[(3)] If $n=6, 7$, and $\mathring{R} \in \mathcal{C}(\alpha,\theta(n,\alpha) )$ for some $\alpha \in (1,\frac{n^4-n^3+8 n-8}{3 n^3+5 n^2-22 n+8}]$;
    \item[(4)] If $n =8, 9, 10$, and $\mathring{R} \in \mathcal{C}(\alpha,\theta(n,\alpha) )$ for some $\alpha \in (2,\frac{n^4-n^3+8 n-8}{3 n^3+5 n^2-22 n+8}]$;
    \item[(5)] If $11\le n \le 16$, and $\mathring{R} \in \mathcal{C}(\alpha,\theta(n,\alpha) )$ for some $\alpha \in ([\frac{n+2}{4}],\frac{n^4-n^3+8 n-8}{3 n^3+5 n^2-22 n+8}]$;
    \item[(6)] If $n \ge 17$, and $\mathring{R} \in \mathcal{C}(\alpha,\theta(n,\alpha) )$ for some $\alpha \in ([\frac{n+2}{4}],\frac{n^2+n-8}{4n-8}]$.
\end{itemize}
Here $\theta(n,\alpha)$ denotes the explicit constant defined in \eqref{theta}.
\end{corollary}



\begin{remark}
(1) The results of Theorems \ref{thm1} and \ref{thm2} for $\alpha = 1, 2$ correspond to those in \cite{CW24-1, CW25-2}, respectively, while for integers $3 \le \alpha \le [\frac{n+2}{4}]$ they correspond to the results in \cite{FL25}. Theorem \ref{thm1}-\ref{thm2} extend these results from the positive integer case to the real case.\\
(2) 
The cone $\mathcal C(\a,\theta)$ is monotone in both parameters: it expands as $\a$ or $\theta$ increases (see, \cite[Proposition 2.8]{Li25}).
Therefore, Corollary \ref{cor1} extends the results of Dai-Fu \cite{DF24}.
\end{remark}

The paper is organized as follows:  In Section \ref{sec 2}, we give the definition of the curvature operator of the second kind and derive some identities for the curvatures on Einstein manifolds.
Section \ref{sec 3} is devoted to the proof of Theorems \ref{thm1}-\ref{thm2}.

\section{Preliminaries}\label{sec 2}

In this section, we recall some basic properties about the curvature operator of the second kind, and a Weyl tensor formula on Einstein manifolds.
For further details, we refer the reader to \cite{BK78, CV60, CaM20, CW24-1, CW25-2,  Kashiwada93, Li21, Nishikawa86}.

\subsection{Curvature Operator of the Second Kind}\label{sec 2.2}
This subsection is devoted to an introduction to the curvature operator of the second kind.

Consider an $n$-dimensional Riemannian manifold $(M^n, g)$. At a point $p \in M$, set $V = T_p M$ and let $\{e_i\}_{i=1}^n$ be an orthonormal basis of $V$. The metric $g$ allows us to identify $V$ with its dual space $V^*$. Denote by $S^2(V)$ and $\wedge^2(V)$ the spaces of symmetric two-tensors and two-forms on $V$, respectively. The space $S^2(V)$ decomposes into $O(V)$-irreducible subspaces as
$$
S^2(V) = S^2_0(V) \oplus \mathbb{R} g,
$$
where $S^2_0(V)$ is the space of traceless symmetric two-tensors, and $g = \sum_{i=1}^{n} e_i \otimes e_i$. 
Let $N$ be the dimension of $S^2_0(V)$, i.e., $N = (n+2)(n-1)/2$.
The space of symmetric two-tensors on $\wedge^2 V$, denoted by $S^2(\wedge^2 V)$, admits an orthogonal decomposition 
$$
S^2(\wedge^2 V) = S^2_B(\wedge^2 V) \oplus \wedge^4 V,
$$
where $S^2_B(\wedge^2 V)$ is the space of algebraic curvature operators on $V$, consisting of all tensors $R \in S^2(\wedge^2 V)$ that satisfy the first Bianchi identity.

For an algebraic curvature operator $R \in S^2_B(\wedge^2(T_p M))$, there are two associated self-adjoint operators (see \cite{Nishikawa86}).
The curvature operator of the first kind, denoted by $\hat{R}: \wedge^2 (T_p M) \to \wedge^2(T_p M)$, is defined as
\[
\hat{R}(\omega)_{ij}=\frac{1}{2}\sum_{k,l=1}^n R_{ijkl}\omega_{kl}.
\]
The second operator is $\overline{R}: S^2(T_p M) \to S^2(T_p M)$, given by
\[
    \overline{R}(\varphi)_{ij}=\sum_{k,l=1}^n R_{iklj}\varphi_{kl}.
\]
Following Nishikawa \cite{Nishikawa86}, the curvature operator of the second kind refers to the symmetric bilinear form
$$
    \mathring{R}: S^2_0(T_pM) \times S^2_0(T_pM) \to \mathbb{R}
$$
obtained by restricting $\overline{R}$ to the space $S^2_0(T_pM)$ of traceless symmetric two-tensors. Equivalently, as noted in \cite{NPW22}, it can be viewed as the self-adjoint operator
$$
\mathring{R} = \pi \circ \overline{R}: S^2_0(T_pM) \to  S^2_0(T_pM),
$$
where $\pi: S^2(T_pM) \to S^2_0(T_pM)$ denotes the projection.

Consider the eigenvalues $\{\lambda_j\}_{j=1}^N$ of $\mathring{R}$ and denote their average by $\bar{\lambda} = \frac{1}{N}\sum_{j=1}^N \lambda_j$. For an Einstein manifold of dimension $n$, the scalar curvature $\S$ satisfies
\[
\S = n(n-1)\bar{\lambda},
\]
as shown in \cite[Proposition 2.1]{CW24-1}.

\begin{definition}[\cite{NPW22}]
\label{def2.4}
Let $\mathcal{T}^{(0, k)}(V)$ denote the space of $(0, k)$-tensor space on $V$. For $S\in S^{2}(V)$ and $T\in\mathcal{T}^{(0, k)}(V)$, we define 
\begin{align*}
    S:\  &\mathcal{T}^{(0, k)}(V) \to \mathcal{T}^{(0, k)}(V),\\
    &(ST)({X_1}, \cdots, {X_k}) = \sum\limits_{i = 1}^k {T({X_1}, \cdots , S{X_i}, \cdots, {X_k})},
\end{align*}
and define $T^{S^{2}}\in \mathcal{T}^{(0, k)}(V)\otimes S^{2}(V)$ by
\begin{align*}
    \left\langle
    T^{S^{2}}(X_{1}, \cdots, X_{k}), S
    \right\rangle
    =(ST)\left(X_{1}, \cdots,  X_{k}\right).
\end{align*}
\end{definition}
According to the above definition, if $\{\bar S^{j}\}_{j=1}^N$
 is an orthonormal basis for $S^{2}(V)$, then 
$$T^{S^{2}}=\sum\limits_{j=1}^N\bar S^{j}T\otimes\bar S^{j}.$$
Similarly, we define $T^{S^{2}_{0}}\in \mathcal{T}^{(0, k)}(V)\otimes S^{2}_{0}(V)$ by
$$T^{S^{2}_{0}}=\sum\limits_{j=1}^N S^{j}T\otimes S^{j},$$
where $\{ S^{j}\}_{j=1}^N$ is an orthonormal basis for $S^{2}_{0}(V)$.

\subsection{A formula for Weyl tensor on Einstein manifolds}\label{sec:Weyltensor}
For any symmetric two-tensors $A, B \in S^2(V)$, the Kulkarni–Nomizu product $A \owedge B$ yields an algebraic curvature operator in $S^2_B(\wedge^2 V)$, defined by
$$
(A \owedge B)_{ijkl} = A_{ik}B_{jl}+A_{jl}B_{ik} -A_{jk}B_{il}-A_{il}B_{jk}.
$$
The Riemann curvature tensor $R$ admits a decomposition into irreducible components (see \cite[(1.79)]{CaM20}):
\begin{align}\label{2.3}
R = W + \frac{1}{n-2}\,\Ric \owedge g - \frac{\S}{2(n-1)(n-2)}\,g \owedge g,
\end{align}
where $W$ denotes the Weyl tensor and $\Ric$ the Ricci tensor.  
In an arbitrary basis, this reads
\begin{align*}
R_{ijkl}=& \; W_{ijkl}
+ \frac{1}{n-2}\bigl(R_{ik}g_{jl}+R_{jl}g_{ik}-R_{il}g_{jk}-R_{jk}g_{il}\bigr)\\
&-\frac{\S}{(n-1)(n-2)}\bigl(g_{ik}g_{jl}-g_{il}g_{jk}\bigr).
\end{align*}
If $(M^n,g)$ is Einstein, then $\Ric = \frac{\S}{n}\,g$, and \eqref{2.3} reduces to
\begin{align}\label{2.4}
R = W + \frac{\S}{2n(n-1)}\,g \owedge g.
\end{align}

\section{Proof of Theorems \ref{thm1} and \ref{thm2}}\label{sec 3}
In this section, we prove the main theorems of this paper. The key to proving Theorems \ref{thm1} and \ref{thm2} is to establish the inequality $\left\langle \Delta R, R\right\rangle\ge 0$ for the Riemann curvature tensor. Following \cite{CW24-1, CW25-2}, we begin by expressing $\left\langle \Delta R, R\right\rangle$ for Einstein manifolds under the cone condition 
\[
\alpha^{-1}\big(\sum_{i=1}^{[\alpha]}
\lambda_i+ (\alpha - [\alpha] ) \lambda_{\lfloor \alpha \rfloor + 1} \big) \ge -\theta \bar{\lambda},
\]
in terms of the curvature operator of the second kind $\mathring{R}$, where $\lambda_1 \le \lambda_2 \le \cdots \le \lambda_N$ be the eigenvalues of $\mathring{R}$ and $\bar{\lambda} = \frac{1}{N}\sum_{i=1}^N \lambda_i$ their average.



\begin{lemma}\label{lm3.1}
    Let $(M^n, g)$ be an  Einstein manifold of dimension $n\ge 6$. Assume that the cone condition \eqref{cone-cond} holds for $\theta\ge 0$ and $1 \le \alpha\le \frac{n^2+n-8}{4(n-2)}$. 
    Then
   \begin{align}\label{4.3}
    \sum_{j=1}^N
    \lambda_{j} |S^{j} W|^2 
    \ge -\frac{16(N-3)}{3n}\theta \bar\lambda \sum_{j=1}^N \lambda_j^2 
    +\frac{16N(N-3)}{3n} \theta \bar\lambda^3,
\end{align}
where $S^{j}W$ is given in Definition \ref{def2.4}.
\end{lemma}
\begin{proof}
   Denote by  $$|S^{\b} W|=\max\limits_{1\le j \le N}|S^{j} W|.$$ 
Using $\l_1\le \l_2\le\cdots\le \l_N$, we estimate  
\begin{align}\label{3-4}
\begin{split}
      \sum_{j=1}^N
    \lambda_{j} |S^{j} W|^2
    \ge& \sum_{j=1}^{[\a]}\l_j|S^{j} W|^2+\l_{[\alpha]+1}\sum_{j={[\alpha]}+1}^N|S^{j} W|^2\\
    =&\sum_{j=1}^{[\a]}(\l_j-\l_{[\a]+1})|S^j W|^2+\l_{[\a]+1}\sum_{j=1}^{N}|S^j W|^2\\
    \ge&|S^\b W|^2\sum_{j=1}^{[\a]}(\l_i-\l_{[\a]+1})+\l_{[\a]+1}\sum_{j=1}^{N}|S^j W|^2\\
    =&|S^\b W|^2\left(\sum_{j=1}^{[\a]}\l_j+(\a-[\a])\l_{[\a]+1}\right)\\
    &+\l_{[\a]+1}\left(\sum_{j=1}^{N}|S^j W|^2-\a|S^\b W|^2\right).
\end{split}
  \end{align}
Recall from  (4.1) and (4.5) of \cite{DF24} that 
\begin{align}\label{DF-4.1}
\sum_{j=1}^{N}|S^jW|^2=\frac{2(n^2+n-8)}{n}|W|^2=\frac{4N-12}{n}|W|^2,
\end{align}
and 
\begin{align}\label{DF-4.5}
|S^jW|^2\le\frac{8n-16}{n}|W|^2, \quad 1\le j\le N.
\end{align}
\eqref{cone-cond} and \eqref{DF-4.5} imlpy that
\begin{align}\label{eq-SW1}
    |S^\b W|^2\left(\sum_{j=1}^{[\a]}\l_j+(\a-[\a])\l_{[\a]+1}\right)\ge-\a\theta\bar{\l}|S^\b W|^2\ge-\a\theta\bar\l\frac{8n-16}{n}|W|^2.
\end{align}
Through \eqref{DF-4.1}, \eqref{DF-4.5} and $\a\le\frac{n^2+n-8}{4(n-2)}$, we conclude that
\begin{align}\label{eq-SW2}
    \sum_{j=1}^{N}|S^j W|^2-\a|S^\b W|^2\ge\frac{4(N-3)-8(n-2)\a}{n}|W|^2\ge0.
\end{align}
Monotonicity of $\{\lambda_i\}_{i=1}^N$ and \eqref{cone-cond} yield $\lambda_{[\alpha]+1} \ge -\theta\bar{\lambda}$. Then, combining this with \eqref{eq-SW2}, we obtain
\begin{align}\label{eq-SW3}
    \l_{[\a]+1}(\sum_{j=1}^{N}|S^j W|^2-\a|S^\b W|^2)\ge-\theta\bar\l\frac{4(N-3)-8(n-2)\a}{n}|W|^2.
\end{align}
From \cite[Lemma 3.1]{CW25-2}, we have
\begin{align}  \label{3-7}  
|W|^2    
=\frac{4}{3} \sum_{j=1}^N \lambda_j^2    
-\frac{4N}{3}\bar\lambda^2.
\end{align}
Inequality \eqref{4.3} is a consequence of \eqref{eq-SW1}, \eqref{eq-SW3} and \eqref{3-7}.
\end{proof}

Recall from \cite[Lemma 3.3]{CW25-2} that
\begin{align*}
    3\langle \Delta R,R\rangle
  =&\sum_{j=1}^N {{\lambda _j }{{\left| {{S^j }W} \right|}^2}}
  -\frac{16N(2N-9n+6)}{3n}\bar\lambda^3\\
   &+\frac{16(2N-12n+6)}{3n} \bar\lambda  \sum_{j=1}^N \lambda _j ^2  + 16\sum_{j=1}^N {\lambda _j ^3}.
\end{align*}
Using  Lemma \ref{lm3.1}, when $n\ge 6$, we have the following estimate
\begin{equation}
\begin{aligned}\label{targetfcn}
    3\langle \Delta R,R\rangle
    \ge & \frac{16}{3n}\Big[ N(N-3)\theta-(2N-9n+6)N \Big]\bar\l^3\\
   &+ \frac{16}{3n} \Big[ (2N-12n+6)-(N-3)\theta \Big] \bar\lambda \sum_{j=1} ^{N}\lambda_{j}^2+16 \sum_{j=1}^{N}\lambda_{j}^3.
\end{aligned}  
\end{equation}

We now present a lemma that provides a lower bound estimate for $\l_1$ under the considered conditions.

\begin{lemma}\label{lm3.2}
Let $\{\l_i\}_{i=1}^N$ satisfy $\sum_{i=1}^{N}\l_i=N\bar\l$, $\l_1\le\l_2\le\cdots\le\l_N$ and the cone condition \eqref{cone-cond} for some $\theta\ge 0$ and $1\le\a<N$. Here $\bar\lambda=\frac 1N \sum_{i=1}^N \l_i$. Then
\begin{align}\label{lem1}
        \l_1\ge-\frac{(\a-1)N+\a(N-1)\theta}{N-\a}\bar\l.
\end{align}
\end{lemma}
\begin{proof}
  We first provide a proof by induction when $\a$ is a positive integer. 
  If $\a=1$, \eqref{lem1} reduces to $\l_1\ge-\theta\bar\l$, which is obvious. Suppose the statement holds for $\a=k-1$, where $k\le N$ is a positive integer. 
It remains to prove that \eqref{lem1} holds when $\a=k$.
The non-decreasing property of $\lambda_i$ and $\sum_{i=1}^{N}\lambda_i=N\bar{\lambda}$ yield 
\begin{align}\label{lem3}
    N\bar\l\ge(\l_1+\cdots+\l_{k-1})+(N-k+1)\l_k.
\end{align}
   \eqref{cone-cond} and $k\le N$ directly imply
   \begin{align}\label{lem4}
       (N-k+1)(\l_1+\cdots+\l_{k-1})+(N-k+1)\l_k\ge-k(N-k+1)\theta\bar\l.
   \end{align}
   Summing up \eqref{lem3}-\eqref{lem4}, we see
   \begin{equation}\label{lem5}
       \l_1+\cdots+\l_{k-1}\ge-\frac{N+k(N-k+1)\theta}{N-k}\bar\l.
   \end{equation}
The induction hypothesis for $\a=k-1$ and \eqref{lem5} yields
\begin{align}\label{lem6}
    \l_1\ge-\frac{N(k-1)+k(N-1)\theta}{N-k}\bar\l.
\end{align}
This proves that the statement holds for all $\a\in\mathbb{Z}_+$. Now, we consider non-integral $\alpha$ in the interval $[1, N)$. Recall the basic inequality $[\alpha] \le \alpha < [\alpha] + 1$.
Similarly, from $\sum_{i=1}^{N}\l_i=N\bar\l$, $\l_1\le\l_2\le\cdots\le\l_N$ and \eqref{cone-cond}, we derive
\begin{align}\label{lem8}
    (\a-[\a])N\bar\l\ge(\a-[\a])(\l_1+\cdots+\l_{[\a]})+(\a-[\a])(N-[\a])\l_{[\a]+1},
\end{align}
and
\begin{align}\label{lem9}
    (N-[\a])(\l_1+\cdots+\l_{[\a]})+(\a-[\a])(N-[\a])\l_{[\a]+1}\ge-(N-[\a])\a\theta\bar\l.
\end{align}
Adding \eqref{lem8} and \eqref{lem9} gives
\begin{align*}
    \l_1+\l_2+\cdots+\l_{[\a]}\ge-\frac{(\a-[\a])N+(N-[\a])\a\theta}{N-\a}\bar\l.
\end{align*}
The desired inequality follows by applying the conclusion of induction.
\end{proof}
A key step in the proof of the main theorem lies in the following proposition, which establishes the nonnegativity of $\langle \Delta R, R\rangle$ under the given curvature constraints.

\begin{proposition}\label{prop3.3}
Let $(M^n,g)$ be an Einstein manifold of dimension $n\ge 6$ or $n=4, 5$. Suppose that $\mathring{R}$ satisfies the cone condition \eqref{cone-cond} and 
\begin{align*}
1\le \alpha \le 
\begin{cases}
\min\left\{\frac{n^4-n^3+8 n-8}{3 n^3+5 n^2-22 n+8}, \frac{n^2+n-8}{4n-8} \right\}, \quad &n \ge 6,\\
\frac{(n+2)(n+5)}{3n+8}\quad &n=4,5.
\end{cases}
\end{align*}
Then 
\begin{align}\label{k-e}
   \left\langle \Delta R, R\right\rangle\ge 0,
\end{align}
and the equality holds if and only if 
 \begin{align*}
    \l=\left(1, 1,\cdots,1\right)\bar\l \ \ \  \text{or} \ \ \
    \l=\left(-\frac{(\a-1)N+\a(N-1)\theta}{N-\a}, \frac{N+\a\theta}{N-\a}, \cdots, \frac{N+\a\theta}{N-\a}\right)\bar\l.
\end{align*}
Here, $\bar \l=\sum_{j=1}^N \l_j/N$ is the average of $\l_j$.
\end{proposition}
\begin{proof}
We prove the case for $n\ge 6$. Let $\l=(\l_1, \l_2,\cdots,\l_N)$ and  denote the right hand side of \eqref{targetfcn} by $f(\l)$, namely
\begin{align*}
\begin{split}
    f(\l):=&\frac{16N}{3n}\Big[ (N-3)\theta-(2N-9n+6) \Big]\bar\l^3\\
   &+  \frac{16}{3n}\Big[ (2N-12n+6)-(N-3)\theta \Big] \bar\lambda \sum_{j=1} ^{N}\lambda_{j}^2+16 \sum_{j=1}^{N}\lambda_{j}^3,
\end{split}
\end{align*}
where 
\[
\theta(n,\alpha)
= \frac{3(N-n+1)(N-\alpha)}{3n\alpha(N-2)+(N-3)(N-\alpha)}-1.
\]
Using
$$
(N-3)\theta-(2N-9n+6)=-3n\left(\frac{N-2}{N-1}A-\frac{N}{N-1}\right)
$$
and
$$
(2N-12n+6)-(N-3)\theta=3n\left(\frac{N-2}{N-1}A-\frac{2N-1}{N-1}\right),
$$
we have
\begin{align}
    f(\l)=16\left[
    \sum_{j=1}^{N}\lambda_{j}^3
    +\left(\frac{N-2}{N-1}A-\frac{2N-1}{N-1}\right)\bar\l \sum_{j=1} ^{N}\lambda_{j}^2
    -N\left( \frac{N-2}{N-1}A-\frac{N}{N-1} \right)\bar\l^3
    \right],
\end{align}
where $A=\frac{(\a-1)N+\a(N-1)\theta}{N-\a}$ is the lower bound estimates obtained in Lemma \ref{lm3.2}. 
The curvature condition \eqref{theta} and $\theta>-1$ show that $\bar \l\ge 0$. If $\bar \l=0$, then $\l_1=\cdots=\l_N=0$, and \eqref{k-e} holds trivially due to \eqref{targetfcn}. For the case $\bar \l>0$, let
 \begin{align*}
   x_j=\l_j/\bar\l+A,\quad  x=(x_1, \cdots, x_N),
 \end{align*}
then $\sum_{j=1}^N x_j =N(1+A)$ and $x_j \ge 0$. 
This allows us to express $f(\l)$ as
\begin{align*}
f(\l)=16 \bar\l^3 F(x),
\end{align*}
where
\begin{align*}
F(x)=\sum_{i=1}^{N}x_i^3-(3-\frac{N-2}{N-1})(1+A)\sum_{i=1}^Nx_i^2+\frac{(A+1)^3 N^2}{N-1}.
\end{align*}

Now we seek the minimal points of $F(x)$ with constraints
\[
\sum_{j=1}^N x_j =N(1+A), \quad x_j \ge 0.
\]
Following the same method as in \cite[Proposition 3.4]{CW24-1}, we have $F\ge 0$, and the minimizers being
\[
x=(1,\cdots,1) \quad \text{and} \quad x=(0,\frac{N(1+A)}{N-1},\cdots,\frac{N(1+A)}{N-1}).
\]
Since the relaxed constraints lies within the feasible region of the original problem, then it is indeed an optimal point for the original problem.
For the case $n=4,5$, according to \cite[Section 4.2]{DF24}, $\langle \Delta R,R\rangle$ can be expressed explicitly in terms of the curvature operator of the second kind as:
\[
\langle \Delta R,R\rangle= 8 \sum_{j=1}^N \l_j^3 +\frac{8(n-4)}{3}\bar\l \sum_{j=1}^N \l_j^2 -\frac{4(n+2)(n-1)^2}{3}\bar\l^3.
\]
By an argument analogous to that for $n\ge 6$, one can establish that Proposition \ref{prop3.3} holds for $n=4,5$.
\end{proof}

Now we prove our main theorems.

\begin{proof}[Proof of Theorem \ref{thm1}]
By \eqref{k-e}, we have
\begin{align*}
     \Delta |R|^2
    =2|\nabla R|^2 + 2\langle \Delta R,R \rangle \ge 2|\nabla R|^2.
\end{align*}
Integrating this inequality over $M$ yields
\begin{align*}
    0=\int_M  \Delta |R|^2\, d\mu_g \ge\int_M 2|\nabla R|^2 \, d\mu_g \ge 0,
\end{align*}
which implies $|\nabla R|=0$. Therefore, $M$ is a symmetric space. \eqref{theta} and \cite[Theorem 1.8]{Li25} imply that $M$ is either flat or a rational homology sphere. Compact symmetric spaces that are rational homology spheres were classified completely by Wolf \cite[Theorem 1]{Wolf}. Apart from spheres, the only simply connected example is $SU(3)/SO(3)$. Proposition \ref{prop3.3} implies that the eigenvalues of $\mathring{R}$ are either $\left(1, 1,\cdots,1\right)\bar\l$ or $\left(-\frac{(\a-1)N+\a(N-1)\theta}{N-\a}, \frac{N+\a\theta}{N-\a}, \cdots, \frac{N+\a\theta}{N-\a}\right)\bar\l$. Combining this with \cite[Example 4.5]{NPW22}, we know that $M$ is a round sphere if $\bar \l>0$.
\end{proof}

\begin{proof}[Proof of Theorem \ref{thm2}]
Analogous to the proof of Theorem \ref{thm1}, Theorem \ref{thm2} is a direct consequence of Proposition \ref{prop3.3} and \cite[Theorem B(b)]{NPW22}. 
\end{proof}

\section*{Acknowledgments}
The first author expresses her gratitude to her advisor, Professor Ying Zhang, for lots of encouragement and helpful suggestions.  The research of this paper is partially supported by  NSF of Jiangsu Province Grant No.BK20231309.

\bibliographystyle{plain}
\bibliography{ref}
\end{document}